\documentclass[10pt]{article}
\usepackage{pifont}

\usepackage{amsmath}
\usepackage{amscd}
\usepackage{amsfonts}
\usepackage{latexsym}
\usepackage{amssymb}
\usepackage{stmaryrd}
\usepackage{overpic}
\usepackage{graphicx}
\usepackage{subcaption}
\usepackage[unicode,colorlinks,linkcolor=blue,hyperindex]{hyperref}

\usepackage{tikz}
\usetikzlibrary{er}

\newtheorem{theorem}{Theorem}[section]

\newtheorem{lemma}[theorem]{Lemma}

\newtheorem{remark}[theorem]{Remark}
\newenvironment{proof}[1][Proof]{\textbf{#1.} }
{\ \rule{0.75em}{0.75em}\smallskip}

\numberwithin{equation}{section}

\newcommand{\vect}[1]{\boldsymbol{#1}} 

\begin{document}

\title{A Mixed Discontinuous Galerkin Method for Linear Elasticity\\
with Strongly Imposed Symmetry\thanks{The work of Fei Wang is 
partially supported by the National Natural Science Foundation of
China (Grant No.\ 11771350).  The work of Shuonan Wu is partially
supported by the startup grant from Peking University. The work of the
Jinchao Xu is partially supported by US Department of Energy Grant
DE-SC0014400 and National Science Foundation grant DMS-1819157.}}

\author{Fei Wang\footnote{feiwang.xjtu@xjtu.edu.cn, School of
Mathematics and Statistics \& State Key Laboratory of Multiphase Flow
in Power Engineering, Xi'an Jiaotong University, Xi'an, Shaanxi
710049, China} 
\quad 
Shuonan Wu\footnote{snwu@math.pku.edu.cn, School of
Mathematical Sciences, Peking University, Beijing, 100871, China} 
\quad 
Jinchao Xu\footnote{xu@math.psu.edu, Department of Mathematics,
Pennsylvania State University, University Park, PA, 16802, USA}
}
\date{}
\maketitle 

\begin{abstract}
In this paper, we study a mixed discontinuous Galerkin (MDG) method
to solve linear elasticity problem with arbitrary order discontinuous
finite element spaces in $d$-dimension ($d=2,3$).  This method uses
polynomials of degree $k+1$ for the stress and of degree $k$ for the
displacement ($k\geq 0$).  The mixed DG scheme is proved to be
well-posed under proper norms. Specifically, we prove that, for any $k
\geq 0$, the $H({\rm div})$-like error estimate for the stress and
$L^2$ error estimate for the displacement are optimal. We further
establish the optimal $L^2$ error estimate for the stress provided
that the $\mathcal{P}_{k+2}-\mathcal{P}_{k+1}^{-1}$ Stokes pair is
stable and $k \geq d$. We also provide numerical results of MDG
showing that the orders of convergence are actually sharp.
\end{abstract}




\section{Introduction}
In this paper, we present a mixed discontinuous Galerkin (MDG) method
for the following linear elasticity problem:
\begin{equation} \label{equ:elasticity}
\left\{
\begin{aligned}
\mathcal{A} \vect{\sigma} - \vect{\varepsilon}(u) &= 0 &
\text{in~}\Omega, \\
{\rm div} \vect{\sigma}&= f & \text{in~} \Omega, \\
u &= 0 & \text{on~} \partial\Omega, 
\end{aligned}
\right.
\end{equation}
where $u: \Omega \mapsto \mathbb{R}^d$ and $\vect{\sigma}: \Omega
\mapsto \mathbb{S}$, denote displacement and stress, respectively.
Here, $\mathbb{S}$ represents the space of real symmetric matrices of
order $d \times d$. The tensor $\mathcal{A}: \mathbb{S} \mapsto
\mathbb{S}$ is assumed to be bounded and symmetric positive definite,
and the linearized strain tensor is denoted by $\vect{\varepsilon}(u)
= (\nabla u + (\nabla u)^t)/2$. 

For the mixed methods for linear elasticity problem
\eqref{equ:elasticity}, it is very challenging to develop the stable
mixed finite element methods because the stress tensor needs to be
symmetric. One approach to circumvent this difficulty is to introduce
the antisymmetric part of $\nabla u$ as a new variable, and hence, to 
enforce stress symmetry weakly \cite{amara1979equilibrium,
arnold2007mixed, boffi2009reduced, cockburn2010new, farhloul1997dual,
qiu2009mixed, gopalakrishnan2012second}.  Another approach is to use
the composite element for the stress \cite{johnson1978some,
arnold1984family}.   The first stable non-composite conforming mixed
finite element method for plane elasticity was proposed by Arnold and
Winther in 2002 \cite{arnold2002mixed}, and analogs of the results
in the 3D case were reported in \cite{adams2005mixed, arnold2008finite}.
In this class of elements, the displacement is discretized by
discontinuous piecewise $\mathcal{P}_k^{-1}$ ($k\geq 1$) polynomial,
while the stress is discretized by the conforming $\mathcal{P}_{k+2}$
tensors whose divergence is $\mathcal{P}_k$ vector on each triangle.  
In recent years, Hu and Zhang \cite{hu2014family, hu2015family} and Hu
\cite{hu2015finite} proposed a family of conforming mixed elements for
$\mathbb{R}^d$ that apply the $\mathcal{P}_{k+1}-\mathcal{P}_k$ pair
for the stress and displacement when $k \geq d$. These elements also
admit a unified theory and a relatively easy implementation. The lower
order conforming approximations of stress were also considered in
\cite{hu2016finite}, and a simpler stress element with jump
stabilization term for the displacement \cite{chen2017stabilized}. 

Because of the lack of suitable conforming mixed elasticity elements,
several authors have resorted to the nonconforming elements 
\cite{arnold2003nonconforming,arnold2014nonconforming,
gopalakrishnan2011symmetric}, where the optimal convergence order for
the displacement can be proved under the full elliptic regularity
assumption but the convergence order of $L^2$ error for stress is
still suboptimal. To improve the convergence order for stress, an
interior penalty mixed finite element method using Crouzeix-Raviart
nonconforming linear element to approximate each component of the
symmetric stress was studied in \cite{cai2005mixed}. In
\cite{wu2017interior}, Wu, Gong, and Xu proposed two classes of interior
penalty mixed finite elements for linear elasticity of arbitrary order
in arbitrary dimension, where the stability is guaranteed by
introducing the  nonconforming face-bubble spaces based on the local
decomposition of discrete symmetric tensors.  

Discontinuous Galerkin (DG) methods have been applied to solve various
differential equations due to their flexibility in constructing
feasible local shape function spaces and the advantage to capture
non-smooth or oscillatory solutions effectively.  The DG methods are
attracting the interest of many applied mathematicians and engineers
because they discretize the equations in an element-by-element
fashion, and glue each element through numerical traces, which can
give rise to locally conservative methods. In
\cite{arnold2002unified}, Arnold, Brezzi, Cockburn, and Marini proposed
a unified framework for the devising and analysis of most DG methods
for second-order elliptic equations.  The LDG method, which is
introduced in \cite{cockburn1998local}, is one of several
discontinuous Galerkin methods which are being vigorously studied
\cite{castillo2000priori, arnold2002unified,
cockburn2003discontinuous, cockburn2009unified}. As proposed in
\cite[Equ. (2.4)]{castillo2000priori}, the numerical traces for
second-order elliptic equations have the general expressions as 
$$ 
\begin{aligned}
\widehat{\vect{p}} &= \{\vect{p}\} - C_{11}\llbracket u \rrbracket
-\vect{C_{12}}[\vect{p}], \\
\widehat{u} &= \{u\} + \vect{C_{12}}\cdot \llbracket u \rrbracket -
C_{22}[\vect{p}],  
\end{aligned}
$$ 
where $u$ and $\vect{p}$ are the approximations of primal variable and
flux, respectively. In most literature, the parameter $C_{22}$ is
taken as $0$ or $\mathcal{O}(h)$ so that the resulting scheme is of
the category of primal DG method. When taking $C_{22}$ as
$\mathcal{O}(h^{-1})$, the penalty term on the jump of $\vect{p}$
leads to a mixed DG scheme \cite{hong2018unified}.

For linear elasticity problem, a primal LDG method was studied in
\cite{chen2010local}, where the discontinuous
$\mathcal{P}_{k}^{-1}-\mathcal{P}_{k+1}^{-1}$ pairs were used to
approximate the stress and displacement for $k\geq 0$.  In the weak
formulation, two penalty terms for stress and displacement are
adopted, however, the error analysis was only given for the case when
the penalty term of the stress vanishes, i.e. $C_{22}=0$.

In this paper, we study the mixed LDG method for solving linear
elasticity by discontinuous
$\mathcal{P}_{k+1}^{-1}-\mathcal{P}_{k}^{-1}$ finite element pairs for
the stress and displacement with $k\geq 0$ for any spatial dimension
in a unified fashion. Our contributions are twofold. First, by
introducing a mesh-dependent norm for the stress, we give a prior
error analysis, which shows that optimal $L^2$-error estimate for
displacement and optimal $H_h({\rm div})$ error estimate for stress.
Second, when the $\mathcal{P}_{k+2}-\mathcal{P}_{k+1}^{-1}$ Stokes
pair is stable and $k\geq d$, we prove the optimal $L^2$ error
estimate for the stress by the BDM projection \cite{brezzi1985two} and
a symmetrization technique. 

The rest of the paper is organized as follows. In Section
\ref{sec:mixed-DG}, we derive the mixed DG scheme to solve the linear
elasticity problem. Then based on Brezzi theory, we prove the
well-posedness of the scheme in Section \ref{sec:well-posedness}, and
the optimal convergence rates are obtained for both stress and
displacement variables in Section \ref{sec:error-estimate}.  In
addition, the optimal $L^2$ error estimate for the stress is shown in
Section \ref{sec:L2error}.  In Section \ref{sec:numerical}, numerical
tests are given for solving the linear elasticity problems by the
mixed LDG methods, and the numerical results verify the theoretical
error analysis. Finally, we give several concluding remarks in the
last section.

\section{Mixed DG method for linear elasticity problem}
\label{sec:mixed-DG}
In this section, we study a mixed discontinuous Galerkin method for
the linear elasticity problem \eqref{equ:elasticity},
whose weak formulation reads: Find $(\vect{\sigma}, u)\in
\vect{\Sigma} \times V$ such that
\begin{equation} \label{equ:weak-formulation}
\left\{
\begin{aligned}
(\mathcal{A} \vect{\sigma}, \vect{\tau})_{\Omega} + (u,{\rm div}
\vect{\tau})_{\Omega} &= 0 & \forall \vect{\tau} \in \vect{\Sigma}, \\
({\rm div} \vect{\sigma}, v)_{\Omega}  &= (f,v)_{\Omega} & \forall
v\in V. \\
\end{aligned}
\right.
\end{equation}
Here, $V = L^2(\Omega;\mathbb{R}^d)$ denotes the space of
vector-valued functions which are square-integrable with the $L^2$
norm, and $\vect{\Sigma} = H({\rm div},\Omega;\mathbb{S})$ consists of
square-integrable symmetric matrix fields with square-integrable
divergence, and the corresponding norm is defined by 
$$
\|\vect{\tau}\|_{{\rm div}, \Omega}^2 := \|\vect{\tau}\|_{0, \Omega}^2
+ \|{\rm div} \vect{\tau}\|_{0, \Omega}^2 \qquad \forall \vect{\tau}
\in H({\rm div}, \Omega; \mathbb{S}).
$$
For the symmetric tensor space $\mathbb{S}$, we define the inner
products by $\vect{\sigma}: \vect{\tau} = \sum_{i,j=1}^d
\sigma_{ij}\tau_{ij}$ for any $\vect{\sigma}, \vect{\tau} \in
\mathbb{S}$. Further, we define the {\em symmetric tensor product}
$\odot$ as 
\begin{equation} \label{equ:tensor-product}
u \odot v := \frac{1}{2}(u\otimes v + v \otimes u) \in \mathbb{S}
\qquad \forall u, v\in \mathbb{R}^d,
\end{equation}
where $u\otimes v$ is a tensor with $u_iv_j$ as its $(i,j)$-th
entry. 

\subsection{DG notation}
We introduce some notation before presenting the mixed DG scheme.
Given a bounded domain $D\subset \mathbb{R}^d$ and a positive integer
$m$, $H^m(D)$ is the Sobolev space with the corresponding usual norm
and semi-norm, which are denoted respectively by $\|\cdot\|_{m,D}$ and
$|\cdot|_{m,D}$. We abbreviate them by $\|\cdot\|_{m}$ and
$|\cdot|_{m}$, respectively, when $D$ is chosen as $\Omega$.  The
$L^2$-inner product on $D$ and $\partial D$ are denoted by $(\cdot,
\cdot)_{D}$ and $\langle\cdot, \cdot\rangle_{\partial D}$,
respectively.  $\|\cdot\|_D$ and $\|\cdot\|_{\partial D}$ are the
norms of Lebesgue spaces $L^2(D)$ and $L^2(\partial D)$, respectively.
We assume $\Omega$ is a polygonal domain and denote by
$\{\mathcal{T}_h\}_h$ a family of triangulations of
$\overline{\Omega}$, with the minimal angle condition satisfied. Let
$h_K ={\rm diam}(K)$ and $h = \max\{h_K: K\in \mathcal{T}_h\}$.
Denote by ${\cal E}_h$ the union of the boundaries of the elements $K$
of $\mathcal{T}_h$, ${\cal E}_h^i$ is the set of interior edges and ${\cal
E}_h^\partial={\cal E}_h\backslash{\cal E}_h^i$ is the set of boundary
edges. Let $e$ be the common edge of two elements $K^+$ and $K^-$, and
$\vect{n}^i$ = $\vect{n}|_{\partial K^i}$ be the unit outward normal
vector on $\partial K^i$ with $i = +,-$.  For any vector-valued
function $v$ and tensor-valued function $\vect{\tau}$, let $v^{\pm}$ =
$v|_{\partial K^{\pm}}$, $\vect{\tau}^{\pm}$ = $\vect{\tau}|_{\partial
  K^{\pm}}$. Then, we define the average $\{\cdot\}$, jump $[\cdot]$ and
tensor jump $\llbracket \cdot \rrbracket$ as follows: 
\begin{align*}
&\{v\} = \frac{1}{2}(v^+ + v^-), \qquad
&\{\vect{\tau}\} = \frac{1}{2}(\vect{\tau}^+ + \vect{\tau}^-) &\qquad {\rm on}\ e\in {\cal E}_h^i,\\
&[\vect{\tau}] = \vect{\tau}^+n^+ + \vect{\tau}^-n^-, \qquad  
&\llbracket v \rrbracket = v^+\odot n^+ + v^-\odot n^- &\qquad {\rm on}\ e\in {\cal E}_h^i,\\
&\{\vect{\tau}\} = \vect{\tau}, \qquad  
&\llbracket v \rrbracket = v\odot n &\qquad {\rm on}\ e\in {\cal E}_h^\partial.
\end{align*}
where $n$ is the outward unit normal vector on $\partial \Omega$.  
Let us give the following identities which are used often in this
section. For any vector-valued function $v$ and
tensor-valued function $\vect{\tau}$, all being continuously
differentiable over $K$, we have the following integration by parts
formula:
\begin{equation} \label{equ:int-parts}
\int_K {\rm div}\vect{\tau} \cdot v \,\mathrm{d}x = 
-\int_K \vect{\tau} : \vect{\varepsilon}(v) \,\mathrm{d}x
+\int_{\partial K}(\vect{\tau} n_K) \cdot v \,\mathrm{d}s, 
\end{equation}
and the following identity:
\begin{equation} \label{equ:dg-identity}
\sum_{K\in\mathcal{T}_{h}} \int_{\partial K} (\vect{\tau} n_K) \cdot
v\,\mathrm{d}s
= \int_{\mathcal{E}_{h}}\{\vect{\tau}\}: \llbracket
v\rrbracket\,\mathrm{d}s +
\int_{\mathcal{E}_{h}^i}[\vect{\tau}]\cdot\{v\}\,\mathrm{d}s.
\end{equation}

Throughout this paper, we shall use letter $C$ to denote a generic
positive constant independent of $h$ which may stand for different
values at its different occurrences. The notation $x \lesssim y$ means
$x \leq Cy$. 
For piecewise smooth vector-valued function $v$ and tensor-valued function
$\vect{\tau}$, let $\nabla_h$ and ${\rm div}_h$ be defined by the relation
$$
\nabla_hv|_K=\nabla v|_K, \quad
{\rm div}_h\vect{\tau}|_K={\rm div} \vect{\tau}|_K,
$$
on any element $K\in{\cal T}_h$, respectively.

\subsection{Mixed LDG scheme}

Now, let us introduce the mixed LDG formulation for
\eqref{equ:elasticity}.  We denote the piecewise vector and symmetric
matrix valued discrete spaces by $V_h$ and $\vect{\Sigma}_h$,
respectively. We multiply \eqref{equ:elasticity} by arbitrary test
functions $\vect{\tau}_h \in \vect{\Sigma}_h$ and $v_h \in V_h$,
respectively, and integration by parts over the element $K \in
\mathcal{T}_h$ to obtain 
\begin{equation}
\label{equ:weak1}
\left\{
\begin{aligned}
\sum_{K\in \mathcal{T}_h} (\mathcal{A}\vect{\sigma},
\vect{\tau}_h)_K + \sum_{K\in \mathcal{T}_h} (u, {\rm div}
\vect{\tau}_h)_K - \sum_{K\in
\mathcal{T}_h}\langle u, \vect{\tau}_h n_K \rangle_{\partial K} & = 0
&\forall \vect{\tau}_h \in \vect{\Sigma}_h,\\
-\sum_{K\in \mathcal{T}_h}  (\vect{\sigma}, \vect{\varepsilon}_h (v_h))_K 
+\sum_{K\in \mathcal{T}_h} \langle {\vect{\sigma}} n_K ,v_h
\rangle_{\partial K} & = \sum_{K\in \mathcal{T}_h} (f,v_h)_K
 &\forall v_h\in V_h.
\end{aligned}
\right.
\end{equation}
Let $\widehat{V}_h$ and $\widehat{\vect{\Sigma}}_h$ be
the piecewise vector and symmetric matrix valued discrete spaces on
$\mathcal{E}_h$, respectively. The approximate solution
$(\vect{\sigma}_h, u_h)$ is then defined by using the weak formulation
\eqref{equ:weak1}, namely 
\begin{equation}
\label{equ:weak-form-sum}
\left\{
\begin{aligned}
\sum_{K\in \mathcal{T}_h} (\mathcal{A}\vect{\sigma}_h,
\vect{\tau}_h)_K + \sum_{K\in \mathcal{T}_h} (u_h, {\rm div}
\vect{\tau}_h)_K - \sum_{K\in
\mathcal{T}_h}\langle \widehat{u}_h, \vect{\tau}_h n_K
\rangle_{\partial K} & = 0
&\forall \vect{\tau}_h \in \vect{\Sigma}_h,\\
-\sum_{K\in \mathcal{T}_h}  (\vect{\sigma}_h, \vect{\varepsilon}_h (v_h))_K  
+\sum_{K\in \mathcal{T}_h} \langle \widehat{\vect{\sigma}}_h n_K ,v_h
\rangle_{\partial K} & = \sum_{K\in \mathcal{T}_h} (f,v_h)_K
 &\forall v_h\in V_h,
\end{aligned}
\right.
\end{equation}
where the numerical traces $\widehat{u}_h \in \widehat{V}_h$ and
$\widehat{\vect \sigma}_h \in \widehat{\vect \Sigma}_h$ need to be
suitably defined to ensure the stability of the method and to enhance
its accuracy.  By the identity \eqref{equ:dg-identity} and integration
by parts \eqref{equ:int-parts}, we get from \eqref{equ:weak-form-sum}
that 
\begin{equation}
\left\{
\begin{aligned}
\int_\Omega \mathcal{A}\vect{\sigma}_h : \vect{\tau}_h\,\mathrm{d}x 
+ \int_\Omega u_h\cdot {\rm div}_h \vect{\tau}_h\,\mathrm{d}x 
- \int_{{\cal E}_h} \llbracket \widehat{u}_h \rrbracket:
\{\vect{\tau}_h\} \,\mathrm{d}s 
- \int_{{\cal E}_h^i} \{\widehat{u}_h\}\cdot [\vect{\tau}_h]
\,\mathrm{d}s &=
0 & \forall \vect{\tau}_h \in \vect{\Sigma}_h, \\
\int_\Omega {\rm div}_h\vect{\sigma}_h \cdot v_h\,\mathrm{d}x 
+ \int_{{\cal E}_h} \{\widehat{\vect{\sigma}}_h - \vect{\sigma}_h\}
: \llbracket v_h \rrbracket\,\mathrm{d}s 
+ \int_{{\cal E}_h^i}  [\widehat{\vect{\sigma}}_h -
\vect{\sigma}_h]\cdot \{v_h\}\,\mathrm{d}s 
&= \int_\Omega f \cdot v_h\,\mathrm{d}x & \forall v_h \in V_h. 
\end{aligned}
\right.
\end{equation}
Similar to the discussion for Poisson problem in
\cite{hong2018unified}, we choose mixed LDG numerical traces as
follow:
\begin{equation} \label{equ:numerical-flux}
\left\{
\begin{aligned}
\widehat{u}_h & =\{u_h\} - \eta [\vect{\sigma}_h] \quad {\rm on}\
\mathcal{E}_h^i,
\quad\quad  \widehat{u}_h=0 \quad {\rm on}\ \mathcal{E}_h^\partial,\\
\widehat{\vect{\sigma}}_h &= \{\vect{\sigma}_h\} ~\qquad\qquad {\rm on}\
\mathcal{E}_h.
\end{aligned}
\right.
\end{equation}
In such choice, it is easy to see that the numerical traces are single
valued. Further, we can see that if $u_h$ and $\vect{\sigma}_h$ are
replaced by the exact solution $u$ and $\vect{\sigma}$, then
$\widehat{u}_h = u|_{{\cal E}_h}$ and
$\widehat{\vect{\sigma}}_h=\vect{\sigma}|_{{\cal E}_h}$ on ${\cal
E}_h$. That is, the numerical traces are consistent. Moreover, we have 
$$ 
\llbracket \widehat{u}_h \rrbracket = \vect{0}, \quad
[\widehat{\vect{\sigma}}_h] = 0, \quad {\rm and} \quad \{
\widehat{\vect{\sigma}}_h - \vect{\sigma}_h \} = 0.
$$ 
Then, we obtain the mixed LDG formulation for \eqref{equ:elasticity}:
Find $(\vect{\sigma}_h, u_h) \in \vect{\Sigma}_h\times V_h$ such that
\begin{equation} \label{equ:mixed-DG}
\left\{
\begin{aligned}
a_h(\vect{\sigma}_h, \vect{\tau}_h) + b_h(\vect{\tau}_h, u_h) &= 0 
& \forall \vect{\tau}_h \in \vect{\Sigma}_h, \\
b_h(\vect{\sigma}_h, v_h) &= (f, v_h)_{\Omega}  &
\forall v_h\in V_h.
\end{aligned}
\right.
\end{equation}
Here, we choose $\eta = \eta_e h_e^{-1}, \eta_e = \mathcal{O}(1)$, and
define  
\begin{subequations}
\begin{align}
a_h(\vect{\sigma}, \vect{\tau}) &= \int_\Omega
\mathcal{A}\vect{\sigma}\cdot \vect{\tau} \,\mathrm{d}x + \int_{{\cal E}_h^i}
\eta_e h_e^{-1} [\vect{\sigma}]\cdot[\vect{\tau}]\,\mathrm{d}s &
\forall \vect{\sigma}, \vect{\tau} \in \vect{\Sigma}_h \cup
\vect{\vect{\Sigma}}, \label{equ:blinear_a}\\
b_h(\vect{\tau}, v) &= \int_\Omega {\rm div}_h \vect{\tau} \cdot
v\,\mathrm{d}x 
- \int_{{\cal E}_h^i} [\vect{\tau}] \cdot \{v\}\,\mathrm{d}s &  
\forall \vect{\tau} \in \vect{\Sigma}_h \cup \vect{\vect{\Sigma}},~
v\in V_h\cup V.
\label{equ:blinear_b}
\end{align}
\end{subequations}
Moreover, we define the following star norm
\begin{equation} \label{equ:star-norm}
\|\vect{\tau}\|_{*,\Omega}^2 := \int_{\Omega} (|\vect{\tau}|^2 +
|{\rm div}_h\vect{\tau}|^2) \,\mathrm{d}x + \int_{{\cal E}_h^i} \eta_e h_e^{-1}
|[\vect{\tau}]|^2 \,\mathrm{d}s \quad \forall \vect{\tau} \in \vect{\Sigma}_h
\cup \vect{\vect{\Sigma}}.
\end{equation}

In the following subsections, we prove the boundedness, stability
and consistency of the mixed LDG formulation \eqref{equ:mixed-DG}
when choosing
\begin{equation} \label{equ:spaces-elasticity}
\begin{aligned}
V_h = V_h^k &= \{v_h \in L^2(\Omega;\mathbb{R}^d): ~v_h |_K \in
  \mathcal{P}_k(K;\mathbb{R}^d) ~\forall K\in\mathcal{T}_h \},\\
\vect{\Sigma}_h=\vect{\Sigma}_h^{k+1} &= \{\vect{\tau}_h\in L^2(\Omega; \mathbb{S}):
  ~\vect{\tau}_h |_K\in \mathcal{P}_{k+1}(K;\mathbb{S}) ~
    \forall K\in\mathcal{T}_h\},
\end{aligned}
\end{equation}
for $k \geq 0$, which lead to the optimal order of convergence. 


\section{Well-posedness of the mixed LDG method}
\label{sec:well-posedness}

The well-posedness of the mixed LDG methods \eqref{equ:mixed-DG}
comes from the boundedness and the stability. 
\paragraph{Boundedness.}
It is easy to check by Cauchy-Schwarz inequality that $a_h(\cdot,
\cdot)$ satisfies 
\begin{equation} \label{equ:boundedness-a}
a_h(\vect{\sigma}, \vect{\tau}) \lesssim \|\vect{\sigma}\|_{*,\Omega}
\|\vect{\tau}\|_{*,\Omega} \quad \forall \vect{\sigma},
  \vect{\tau} \in \vect{\Sigma}_h \cup \vect{\Sigma}.
\end{equation}
The remaining task is the boundedness of $b_h(\cdot, \cdot)$.  To this
end, let us recall the {\em lifting operator} $r_e :
(L^2({\cal E}_h))^d \rightarrow V_h$ defined by
\begin{equation} \label{equ:lifting}
\int_\Omega r_e(w)\cdot v_h \,\mathrm{d}x = -\int_{e}
w \cdot \{v_h\} \,\mathrm{d}s\quad  \forall v_h \in V_h. 
\end{equation}
Then, we have the following lemma (see also \cite{arnold2002unified,
brezzi2000discontinuous}).
\begin{lemma}\label{lem:lifting}
For any edge $e\in \partial K$, it holds 
\begin{equation} \label{equ:lifting-boundedness} 
\|r_e(w)\|_{0,\Omega} \lesssim  h_e^{-1/2} \|w\|_{0,e}.
\end{equation}
\end{lemma}
\begin{proof}
By taking $v_h = r_e(w)$ in \eqref{equ:lifting} and applying the
inverse inequality, we obtain 
$$ 
\| r_e(w) \|_{0,\Omega}^2 \leq \frac{1}{2} \|w\|_{0,e}
(\|r_e(w)^+\|_{0,e} + \|r_e(w)^-\|_{0,e}) \lesssim  h_e^{-1/2}
\|w\|_{0,e} \| r_e(w) \|_{0,\Omega},
$$
which gives rise to \eqref{equ:lifting-boundedness}.
\end{proof}

\begin{lemma} \label{lem:boundedness-b}
It holds that 
\begin{align}
b_h(\vect{\tau}, v_h) &\lesssim \|\vect{\tau}\|_{*,\Omega}
\|v_h\|_{0,\Omega} & \forall \vect{\tau} \in \vect{\Sigma}_h \cup
\vect{\Sigma},\; \forall v_h \in V_h, \label{equ:boundedness-b1}\\
b_h(\vect{\tau}, v) &\lesssim \|\vect{\tau}\|_{*,\Omega}
(\|v\|_{0,\Omega} + h|v|_{1,\Omega,h}) &\forall \vect{\tau} \in
\vect{\Sigma}_h \cup \vect{\Sigma},\; \forall v \in V \cap
H^1(\Omega;\mathbb{R}^d).
\label{equ:boundedness-b2}
\end{align}
\end{lemma}
\begin{proof}
In light of Lemma \ref{lem:lifting}, we have for any $v_h \in V_h$
\begin{equation*} \label{equ:boundedness-b}
\begin{aligned}
b_h(\vect{\tau}, v_h) 
& = \int_{\Omega} \left( {\rm div}_h \vect{\tau} + \sum_{e\in {\cal
    E}_h^i} r_e([\vect{\tau}]) \right) \cdot v_h \,\mathrm{d}x \\
& \lesssim  \|v_h\|_{0,\Omega} \left( \|{\rm div}_h\vect{\tau}\|_{0,\Omega}^2 
+  \int_{{\cal E}_h^i} h_e^{-1} |[\vect{\tau}]|^2 \,\mathrm{d}s\right)^{1/2}\\
& \leq  \|v_h\|_{0,\Omega} \|\vect{\tau}\|_{*,\Omega}.
\end{aligned}
\end{equation*}
Furthermore, for any $v \in V \cap H^1(\Omega;\mathbb{R}^d)$, 
$$ 
\begin{aligned}
b_h(\vect{\tau}, v) & \leq \|{\rm div}_h \vect{\tau}\|_{0,\Omega}
\|v\|_{0,\Omega} + \sum_{e\in \mathcal{E}_h^i}
h_e^{-1/2}\|[\vect{\tau}]\|_{0,e} h_e^{1/2}\|\{v\}\|_{0,e}
\lesssim \|\vect{\tau}\|_{*,\Omega} (\|v\|_{0,\Omega} +
h|v|_{1,\Omega,h}). 
\end{aligned}
$$ 
Here, we use the trace inequality in the last step. 
\end{proof}
  
\paragraph{Stability.}  
According to the theory of mixed method, the stability of the
saddle point problem \eqref{equ:mixed-DG} is the corollary of the
following two conditions \cite{brezzi1974existence,brezzi1991mixed}:
\begin{enumerate}
\item K-ellipticity: There exists a constant $C > 0$, independent of
the grid size such that 
\begin{equation} \label{equ:K-ellipticity}
a_h(\vect{\tau}_h, \vect{\tau}_h) \ge C
\|\vect{\tau}_h\|_{*,\Omega}^2 \qquad \forall \vect{\tau}_h \in Z_h,
\end{equation}
where $Z_h = \{\vect{\tau}_h \in \vect{\Sigma}_h~|~ b_h(\vect{\tau}_h,
    v_h) = 0\; \forall v_h \in V_h\}$.
\item The discrete inf-sup condition: There exists a constant $C > 0$,
  independent of the grid size such that 
\begin{equation} \label{equ:inf-sup}
\inf_{v_h \in V_h} \sup_{\vect{\tau}_h \in \vect{\Sigma}_h}
\frac{b_h(\vect{\tau}_h, v_h)}{\|\vect{\tau}_h\|_{*,\Omega}
  \|v_h\|_{0,\Omega}} \ge C.
\end{equation}
\end{enumerate}
 
First, we prove the inf-sup condition \eqref{equ:inf-sup} in the
following lemma.

\begin{lemma} [Inf-sup condition] \label{lem:inf-sup}
When choosing $\vect{\Sigma}_h \times V_h = \vect{\Sigma}_h^{k+1}
\times V_h^k$ for $k \geq 0$, the discrete inf-sup condition
\eqref{equ:inf-sup} holds true for mixed LDG method 
\eqref{equ:mixed-DG} of linear elasticity problem.
\end{lemma}
\begin{proof}
In \cite{wu2017interior}, Wu, Gong, and Xu introduced a class of
nonconforming finite element spaces for $k\geq 0$ that 
$$ 
\begin{aligned}
\Sigma_{k+1,h}^{(1)} := \{\vect{\tau} ~|~ & \vect{\tau}|_K \in
\mathcal{P}_{k+1}(K;\mathbb{S}), \text{and the moments of $\vect{\tau}n$}
\\
&\text{up to degree $k$ are continuous across the interior edges}\}.
\end{aligned}
$$ 
Thanks to the Lemma 3.3 and Lemma 4.1 in \cite{wu2017interior}, we know
that for any $v_h \in V_h$, there exists a $\bar{\vect{\tau}}_h \in
\Sigma_{k+1,h}^{(1)}$ such that 
\begin{equation} \label{equ:elasticity-inf-sup} 
{\rm div} \bar{\vect{\tau}}_h = v_h\quad {\rm and} \quad
\|\bar{\vect{\tau}}_h\|_{*,\Omega} \lesssim \|v_h\|_{0,\Omega}.
\end{equation} 
Note that $\Sigma_{k+1,h}^{(1)} \subset \vect{\Sigma}_h^{k+1}$ and the
property of $\Sigma_{k+1,h}^{(1)}$ implies that 
$$ 
\int_{{\cal E}_h^i} [\bar{\vect{\tau}}_h] \cdot \{v_h\} \,\mathrm{d}s
= 0 \qquad \forall v_h \in V_h.
$$ 
Here, we use the fact that $\{v_h\}$ is of degree $k$ on the edge.
Therefore, for any $v_h\in V_h^k$ 
$$ 
\sup_{\vect{\tau}_h \in \vect{\Sigma_h^{k+1}}}
\frac{b_h(\vect{\tau}_h, v_h)}{\|\vect{\tau}_h\|_{*,\Omega}} \geq 
\frac{b_h(\bar{\vect{\tau}}_h, v_h)}{\|\bar{\vect{\tau}}_h\|_{*,\Omega}} = 
\frac{\int_\Omega {\rm div} \bar{\vect{\tau}}_h \cdot v_h
\,\mathrm{d}x}{\|\bar{\vect{\tau}}_h\|_{*,\Omega}}
\gtrsim \|v_h\|_{0,\Omega}.
$$ 
Then, we finish the proof.
\end{proof}

\begin{theorem} \label{thm:well-posedness}
The mixed LDG scheme \eqref{equ:mixed-DG} is well-posed for
$(\vect{\Sigma}_h^{k+1}, \|\cdot\|_{*,\Omega})$ and $(V_h^k,
\|\cdot\|_{0,\Omega})$. 
\end{theorem}
\begin{proof}
In light of the boundedness and Lemma \ref{lem:inf-sup}, we only need
to prove the K-ellipticity \eqref{equ:K-ellipticity}. By the
definition of lifting operator \eqref{equ:lifting}, we have 
$$ 
b_h(\vect{\tau}_h, v_h) = \int_\Omega \left( {\rm div}_h\vect{\tau}_h +
\sum_{e\in{\cal E}_h^i} r_e([\vect{\tau}_h]) \right) \cdot v_h
\,\mathrm{d}x, 
$$
which implies that  
$$
Z_h = \{ \vect{\tau}_h\in \vect{\Sigma}_h^{k+1}~|~
{\rm div}_h\vect{\tau}_h + \sum_{e\in{\cal E}_h^i} r_e([\vect{\tau}_h]) =
0\}.
$$
With the help of the Lemma \ref{lem:lifting}, we see that  
$$
\|{\rm div}_h \vect{\tau}_h\|_{0,\Omega} = \|\sum_{e\in{\cal E}_h^i}
r_e([\vect{\tau}_h])\|_{0,\Omega} \lesssim \sum_{e\in{\cal E}_h^i}
h_e^{-1/2} \|[\vect{\tau}_h]\|_{0,e} \qquad \forall \vect{\tau}_h \in
Z_h. 
$$
Let $\eta_0=\inf_{e\in{\cal E}_h^i} \eta_e$ be a positive constant
that independent of the grid size. Then, 
\begin{equation}
a_h(\vect{\tau}_h, \vect{\tau}_h) \geq \|\vect{\tau}_h\|_{0,\Omega}^2 +
\eta_0 \sum_{e\in{\cal E}_h^i}  h_e^{-1} \|[\vect{\tau}_h]\|_{0,e}^2
\gtrsim \|\vect{\tau}_h\|_{*,\Omega}^2 \qquad \forall \vect{\tau}_h
\in Z_h. 
\end{equation}
Then, we finish the proof.
\end{proof}

\begin{remark}
From Lemma \ref{lem:lifting}, we, Gong, can see that the penalty term
$\int_{{\cal E}_h^i} \eta_e h_e^{-1} [\vect{\sigma}_h] \cdot
[\vect{\tau}_h]\,\mathrm{d}s$ can be replaced by $\sum_{e\in{\cal E}_h^i}
\int_\Omega \eta_e r_e([\vect{\sigma}_h]) \cdot r_e([\vect{\tau}_h])
\,\mathrm{d}x$, and the well-posedness of the corresponding scheme can
be proved similarly with a modified norm $\|\vect{\tau}\|_{*,\Omega}^2
:= \int_{\Omega} (|\vect{\tau}|^2 + |{\rm div}_h\vect{\tau}|^2 +
\sum_{e\in{\cal E}_h^i} |r_e([\vect{\tau}])|^2)\,\mathrm{d}x$.
\end{remark}

\section{A priori error estimates in energy norms}
\label{sec:error-estimate}

\begin{lemma}\label{lem:consistency}
Assume the solution $(\vect{\sigma},u)\in \Sigma \times
H^1(\Omega;\mathbb{R}^d)$, we have
\begin{equation} \label{equ:consistency}
\left\{
\begin{aligned}
a_h(\vect{\sigma} - \vect{\sigma}_h, \vect{\tau}_h) + 
b_h(\vect{\tau}_h, u -u_h) & = 0 & \forall \vect{\tau}_h\in
\vect{\Sigma}_h, \\
b_h(\vect{\sigma} - \vect{\sigma}_h, v_h) &= 0 &
\forall v_h\in V_h. 
\end{aligned}
\right.
\end{equation}
\end{lemma}
\begin{proof}
It can be seen that $[\vect{\sigma}] = 0$ and $\llbracket u \rrbracket
= 0$ on ${\cal E}_h^i$ as $(\vect{\sigma}, u)\in \Sigma \times
H^1(\Omega;\mathbb{R}^d)$.  Therefore, 
$$
\begin{aligned}
a_h(\vect{\sigma}, \vect{\tau}_h) + b_h(\vect{\tau}_h, u) = &
\int_\Omega \mathcal{A} \vect{\sigma} : \vect{\tau}_h \,\mathrm{d}x 
+ \int_\Omega u\cdot {\rm div}_h \vect{\tau}_h \,\mathrm{d}x 
- \int_{{\cal E}_h^i} \{u\}\cdot [\vect{\tau}_h] \,\mathrm{d}s \\
= & \int_\Omega \mathcal{A} \vect{\sigma} : \vect{\tau}_h \,\mathrm{d}x 
- \int_\Omega  \vect{\varepsilon}(u) : \vect{\tau}_h \,\mathrm{d}x +
\int_{{\cal E}_h} \llbracket u \rrbracket : \{\vect{\tau}_h\}
\,\mathrm{d}s \\
 = & \int_\Omega \left( \mathcal{A} \vect{\sigma} -
     \vect{\varepsilon}(u)\right) : \vect{\tau}_h \,\mathrm{d}x = 0. 
\end{aligned}
$$
Hence, we prove the first equality in \eqref{equ:consistency}. On the
other hand,  
$$
b_h(\vect{\sigma}, v_h) = \int_\Omega {\rm div} \vect{\sigma} \cdot v_h
\,\mathrm{d}x - \int_{{\cal E}_h^i} [\vect{\sigma}] \cdot \{v_h\}
\,\mathrm{d}s  
= \int_\Omega {\rm div} \vect{\sigma} \cdot v_h\,\mathrm{d}x =
\int_\Omega f \cdot v_h\,\mathrm{d}x,
$$
which implies the second equality in the lemma. 
\end{proof}

By combining Lemma \ref{lem:consistency} and the well-posedness of
mixed LDG formulation \eqref{equ:mixed-DG}, we have the following a
priori error estimates.

\begin{theorem} \label{thm:priori}
Let $(\vect{\sigma}_h, u_h)$ be the solution of the mixed LDG problem
\eqref{equ:mixed-DG}, and $(\vect{\sigma}, u)\in \Sigma \times
H^1(\Omega;\mathbb{R}^d)$ be the solution of \eqref{equ:elasticity}.
Then,
\begin{equation} \label{equ:cea}
\|\vect{\sigma} - \vect{\sigma}_h\|_{*,\Omega} + \|u  - u_h\|_{0,\Omega} 
\lesssim  \inf_{\vect{\tau}_h \in \vect{\Sigma}_h^{k+1}} 
\|\vect{\sigma} - \vect{\tau}_h\|_{*,\Omega}  
+ \inf_{v_h \in V_h^k} (\|u-v_h\|_{0,\Omega} + h|u-v_h|_{1,\Omega,h}). 
\end{equation}
\end{theorem}
\begin{proof}
Define 
$$ 
\mathcal{L}_h (\vect{\tau}_h, v_h; \vect{\theta}_h, w_h) = 
a_h(\vect{\tau}_h, \vect{\theta}_h) 
+ b_h(\vect{\theta}_h, v_h) + b_h(\vect{\tau}_h, w_h),
$$
which satisfies discrete inf-sup condition based on the well-posedness
of \eqref{equ:mixed-DG}. In the light of Lemma \ref{lem:consistency} and the boundedness
\eqref{equ:boundedness-a}, \eqref{equ:boundedness-b1} and
\eqref{equ:boundedness-b2}, we have for any $(\vect{\tau}_h, v_h)\in
\vect{\Sigma}_h^{k+1}\times V_h^k$,
$$
\begin{aligned}
\|\vect{\tau}_h - \vect{\sigma}_h\|_{*,\Omega} + \|v_h -
u_h\|_{0,\Omega} 
& \lesssim  \sup_{(\vect{\theta}_h, w_h)\in \vect{\Sigma}_h^{k+1}
\times V_h^k} \frac{\mathcal{L}_h (\vect{\tau}_h - \vect{\sigma}_h, v_h
- u_h; \vect{\theta}_h, w_h)}{ \|\vect{\theta}_h\|_{*,\Omega} +
  \|w_h\|_{0,\Omega}}\\
& = \sup_{(\vect{\theta}_h, w_h)\in \vect{\Sigma}_h^{k+1}
\times V_h^k} \frac{
a_h(\vect{\tau}_h - \vect{\sigma}, \vect{\theta}_h) + b_h(\vect{\theta}_h,
v_h - u) + b_h(\vect{\tau}_h - \vect{\sigma}, w_h)
}{ \|\vect{\theta}_h\|_{*,\Omega} + \|w_h\|_{0,\Omega}} \\
& \lesssim \|\vect{\tau}_h - \vect{\sigma}\|_{*,\Omega} +
\sup_{\vect{\theta}_h\in \vect{\Sigma}_h^{k+1}} \frac{b_h(\vect{\theta}_h, v_h -
u)}{\|\vect{\theta}_h\|_{*,\Omega}} \\
& \lesssim \|\vect{\tau}_h - \vect{\sigma}\|_{*,\Omega} + \|v_h -
u\|_{0,\Omega} + h|v_h - u|_{1,\Omega,h}.
\end{aligned}
$$ 
By triangle inequality, we finish the proof.
\end{proof}

For $(\vect{\sigma},u) \in  H^{k+2}(\Omega;\mathbb{S}) \times
H^{k+1}(\Omega;\mathbb{R}^d)$, it is well-known that the Scott-Zhang
interpolation \cite{scott1990finite} $I_h^r$ satisfies:  
$$ 
\begin{aligned}
|\vect{\sigma}- I_h^r\vect{\sigma}|_{s,\Omega} &\lesssim
h^{r+1-s}|\vect{\sigma}|_{r+1,\Omega} \qquad 0\leq s \leq r+1 \leq
k+2,\\
|u-I_h^r u|_{s,\Omega} &\lesssim  h^{r+1-s}|u|_{r+1,\Omega}
\qquad 0 \leq s \leq r+1 \leq k+1.
\end{aligned}
$$
Hence, we have the following theorem.

\begin{theorem} \label{thm:error-estimate} 
Assume that the solution of \eqref{equ:elasticity} satisfies
$(\vect{\sigma},u)\in H^{k+2}(\Omega;\mathbb{S}) \times
H^{k+1}(\Omega;\mathbb{R}^d)$.  Then, the solution of the mixed LDG
problem \eqref{equ:mixed-DG} satisfies  
\begin{equation} \label{equ:error-estimate}
\|\vect{\sigma} - \vect{\sigma}_h\|_{*,\Omega} + \|u -
u_h\|_{0,\Omega} \lesssim h^{k+1}(|\vect{\sigma}|_{k+2,\Omega} +
|u|_{k+1,\Omega}).
\end{equation}
\end{theorem}

\section{$L^2$ error estimate of stress}
\label{sec:L2error}

In this section, we prove the optimal $L^2$ error estimate of
$\vect{\sigma}$ provided that the Stokes pair
$\mathcal{P}_{k+2}-\mathcal{P}_{k+1}^{-1}$ is stable and $k \geq d$. 

First, we recall the definition of classical BDM projection
$\Pi_h^{\rm BDM}$ \cite{brezzi1985two}. Given a function $q \in
H({\rm div}, \Omega; \mathbb{R}^d)$, the restriction of $\Pi_h^{\rm BDM}$ to $K$
is defined as the element of $\mathcal{P}_{k+1}(K;\mathbb{R}^d)$ such
that 
\begin{equation} \label{equ:BDM}
\begin{aligned}
\int_e (\Pi_h^{\rm BDM} q - q) \cdot n p_{k+1} \,\mathrm{d}s &= 0
\qquad \forall p_{k+1} \in \mathcal{P}_{k+1}(e), \\
\int_K (\Pi_h^{\rm BDM}q-q) \cdot \nabla p_{k} \,\mathrm{d}x&= 0
\qquad \forall p_{k} \in \mathcal{P}_{k}(K), \\ 
\int_K  (\Pi_h^{\rm BDM} q-q) \cdot p_{k+1}\,\mathrm{d}x&= 0 \qquad
\forall p_{k+1} \in \Phi_{k+1}(K),
\end{aligned}
\end{equation}
where 
$$ 
\Phi_{k+1}(K) = \{v\in \mathcal{P}_{k+1}(K;\mathbb{R}^d):~
  \mathrm{div}v = 0, v\cdot n|_{\partial K} = 0\}.
$$ 
Let $\mathbb{M}$ be the space of real matrices of order $d\times d$.
In light of the BDM projection \eqref{equ:BDM}, on each $K\in
\mathcal{T}_h$, we first define a matrix-valued function
$\widetilde{\vect{\sigma}}_h$ as the only element of
$\mathcal{P}_{k+1}(K;\mathbb{M})$ through the numerical solution
$\vect{\sigma}_h$ and $\widehat{\vect{\sigma}}_h$ in
\eqref{equ:numerical-flux}: 
\begin{equation} \label{equ:row-BDM}
\begin{aligned}
\int_e (\widetilde{\vect{\sigma}}_h - \widehat{\vect \sigma}_h)n \cdot
p_{k+1} \,\mathrm{d}s&= 0
\qquad \forall p_{k+1} \in \mathcal{P}_{k+1}(e;\mathbb{R}^d), \\
\int_K (\widetilde{\vect{\sigma}}_h - \vect{\sigma}_h) : \vect{\nabla}
p_{k} \,\mathrm{d}x&= 0 \qquad \forall p_{k} \in
\mathcal{P}_{k}(K;\mathbb{R}^d), \\ 
\int_K  (\widetilde{\vect{\sigma}}_h-\vect{\sigma}_h) : \vect{p}_{k+1}
\,\mathrm{d}x &= 0 \qquad \forall \vect{p}_{k+1} \in
\vect{\Phi}_{k+1}(K),
\end{aligned}
\end{equation}
where 
$$ 
\vect{\Phi}_{k+1}(K) = \{\vect{\tau}\in \mathcal{P}_{k+1}(K;\mathbb{M}):~
  \mathrm{div}\vect{\tau} = 0, \vect{\tau} n|_{\partial K} = 0\}.
$$ 
Here, the $\vect{\nabla}$ is regarded as the row-wise operator, i.e., 
$$ 
\vect{\nabla}p = \begin{pmatrix}
(\nabla p_1)^t \\ 
\vdots \\
(\nabla p_d)^t
\end{pmatrix}, 
\qquad p = (p_1, \cdots, p_d)^t.
$$ 
Define the following space 
$$ 
{\rm BDM}_{k+1}^{d\times d} := \{\vect{\tau}\in H({\rm
    div},\Omega;\mathbb{M}): ~ \vect{\tau}|_K\in
\mathcal{P}_{k+1}(K;\mathbb{M})~~\forall K \in \mathcal{T}_h\}.
$$ 
Then, we have the following lemma. 
\begin{lemma} \label{lem:tilde-sigma}
The $\widetilde{\vect{\sigma}}_h$ in \eqref{equ:row-BDM} is
well-defined, and 
\begin{subequations} 
\begin{align}
\widetilde{\vect{\sigma}}_h &\in {\rm BDM}_{k+1}^{d\times d},
  \label{equ:div-conforming} \\
\|\widetilde{\vect{\sigma}}_h - \vect{\sigma}_h\|_{L^2(K)} &\lesssim
h_K^{1/2}\|(\widehat{\vect \sigma}_h -
\vect{\sigma}_h)n\|_{L^2(\partial K)}.
\label{equ:row-estimate}
\end{align}
\end{subequations} 
\end{lemma}
\begin{proof}
Since \eqref{equ:row-BDM} can be viewed as the row-wise BDM
projection, then the well-posedness and \eqref{equ:div-conforming}
follows directly by the definition of $\Pi_h^{\rm BDM}$, and by the
fact that the normal component of the numerical trace for the flux is
single-valued. Let $\vect{\delta} = \widetilde{\vect{\sigma}}_h -
\vect{\sigma}_h$, then 
$$
\begin{aligned}
\int_e \vect{\delta}n \cdot p_{k+1}\,\mathrm{d}s &= \int_e
(\widehat{\vect{\sigma}}_h - \vect{\sigma}_h
)n\cdot p_{k+1} \,\mathrm{d}s\qquad \forall p_{k+1} \in
\mathcal{P}_{k+1}(e;\mathbb{R}^d), \\
\int_K \vect{\delta} : \vect{\nabla} p_{k} \,\mathrm{d}x &= 0
~\qquad\qquad\qquad ~~~\qquad\qquad\forall p_{k} \in
\mathcal{P}_{k}(K;\mathbb{R}^d), \\ 
\int_K  \vect{\delta} : \vect{p}_{k+1}\,\mathrm{d}x & = 0 ~~~~\qquad\qquad\qquad\qquad\qquad
\forall \vect{p}_{k+1} \in \vect{\Phi}_{k+1}(K).
\end{aligned}
$$
Then, \eqref{equ:row-estimate} follows easily by the standard scaling
argument; see \cite{boffi2013mixed}.
\end{proof}

Next, we symmetrize $\widetilde{\vect{\sigma}}_h$ by the Stokes pair
$\mathcal{P}_{k+2}- \mathcal{P}_{k+1}^{-1}$.  A similar technique can
be found in \cite{falk2008finite, gopalakrishnan2012second,
gong2018new}.
\begin{lemma} \label{lem:symmetrize}
Suppose that the Stokes pair
$\mathcal{P}_{k+2}-\mathcal{P}_{k+1}^{-1}$ is stable on the grid
$\mathcal{T}_h$. Having $\widetilde{\vect{\sigma}}_h$ defined in
\eqref{equ:row-BDM}, there exists a matrix-valued function
$\widetilde{\vect{\tau}}_h \in {\rm BDM}_{k+1}^{d\times d}$ such that
$\vect{\sigma}_h^\star := \widetilde{\vect{\sigma}}_h +
\widetilde{\vect{\tau}}_h \in H({\rm div},\Omega;\mathbb{S})$, and  
\begin{equation} \label{equ:symmetrize}
\mathrm{div}\widetilde{\vect{\tau}}_h = 0 \quad \text{and} \quad
\|\widetilde{\vect{\tau}}_h\|_{0,\Omega} \lesssim \|\vect{\sigma}_h -
\widetilde{\vect{\sigma}}_h \|_{0,\Omega}.
\end{equation} 
\end{lemma}
\begin{proof}
We construct a divergence-free term $\widetilde{\vect{\tau}}_h =
\vect{\rm curl} \rho_h$ where $\rho_h$ satisfies
\begin{enumerate}
\item For $d=2$: $\rho_h \in H^1(\Omega;\mathbb{R}^2)$ is a vector-valued
function and $\rho_h|_K\in \mathcal{P}_{k+2}(K;\mathbb{R}^2) $;
\item For $d=3$: $\rho_h \in H^1(\Omega; \mathbb{M})$ is a matrix-valued
function and $\rho_h|_K\in \mathcal{P}_{k+2}(K;\mathbb{M}) $.
\end{enumerate}
For the 2D case, the $\mathrm{curl} $ operator is a rotation of the
operator $\nabla$ (i.e., $\mathrm{curl} =(- \partial_y, \partial_x)$)
and applies on each entry of the vector $\rho_h$. For the 3D case, the
$\mathrm{curl}$ operator applies on each row of the matrix $\rho_h$.
By direct calculation, the symmetry of $\widetilde{\vect{\sigma}}_h +
\widetilde{\vect{\tau}}_h$ is equivalent to the following equation, 
\begin{equation}\label{eq:infsup-low1}
\vect{\rm skw}(\vect{\rm curl} \rho_h) = - \vect{\rm
  skw}\widetilde{\vect{\sigma}}_h,
\end{equation}
where $\vect{\rm skw}\vect{\tau} := (\vect{\tau}-\vect{\tau}^T)/2 $.
For a scalar function $v$ or a vector-valued function
$v=(v_1,v_2,v_3)^T$, we further define 
\[
\vect{\rm Skw}_2(v) := \begin{bmatrix} 0 & v\\ -v &0\end{bmatrix} \quad
\text{and} \quad
\vect{\rm Skw}_3(v) := \begin{bmatrix} 0 & v_3 & - v_2 \\ -v_3 & 0 &
v_1 \\ v_2 & -v_1 & 0\\\end{bmatrix}.
\]
Then, the proof can be divided into the following two cases: 
\begin{enumerate}
\item For $n=2$: from \cite{arnold2006finite}, we have $\vect{\rm skw}
(\vect{\rm curl} \rho_h) = \frac{1}{2} \vect{\rm Skw}_2(\mathrm{div}
\rho_h)$. Thus, \eqref{eq:infsup-low1} can be written as:
\begin{equation}\label{eq:infsup-low2D}
\mathrm{div} \rho_h = \widetilde{\sigma}_{h,21} -
\widetilde{\sigma}_{h,12}.
\end{equation}
The stability of Stokes pair
$\mathcal{P}_{k+2}-\mathcal{P}_{k+1}^{-1}$ then implies that there
exists a $\rho_h \in \{ v\in H^1(\Omega;\mathbb{R}^2): ~ v|_K \in
\mathcal{P}_{k+2}(K;\mathbb{R}^2)\}$ satisfying
\eqref{eq:infsup-low2D} and 
$$
\|\rho_h\|_{1,\Omega} \lesssim \|\widetilde{\sigma}_{h,21} -
\widetilde{\sigma}_{h,12}\|_{0,\Omega}
\leq  \| \widetilde{\sigma}_{h,21}  - \sigma_{h,21}\|_{0,\Omega} + \|
\widetilde{\sigma}_{h,12} - \sigma_{h,12}\|_{0,\Omega} \leq
\|\vect{\sigma}_h - \widetilde{\vect{\sigma}}_h\|_{0,\Omega}.
$$

\item For $n=3$: from \cite{arnold2006finite}, we have  
$\vect{\rm skw}(\vect{\rm curl} \rho_h) = -\frac{1}{2}\vect{\rm Skw}_3
(\mathrm{div} ~ \Xi
\rho_h)$, where $\Xi$ is an algebraic operator defined as ${\Xi}\rho_h
= \rho_h^T - \mathrm{tr}(\rho_h) \vect{I}$. Denoting $\eta_h = \Xi
\rho_h$, it is obvious that $\rho_h = \Xi^{-1} \eta_h =  \eta_h^T -
\frac{1}{2}\mathrm{tr}(\eta_h) \vect{I}$. Thus, \eqref{eq:infsup-low1}
can be written as:  
\begin{equation}\label{eq:infsup-low3D}
\mathrm{div} \eta_h = (\widetilde{\sigma}_{h,23}-\widetilde{\sigma}_{h,32},
\widetilde{\sigma}_{h,31}- \widetilde{\sigma}_{h,13},
\widetilde{\sigma}_{h,12}- \widetilde{\sigma}_{h,21})^{T}.
\end{equation}
Again, there exists a $\eta_h \in  \{\tau\in H^1(\Omega;\mathbb{M}):
  ~\tau|_K \in \mathcal{P}_{k+2}(K;\mathbb{M})\}$ satisfying
  \eqref{eq:infsup-low3D} and 
\begin{equation*}
\|\rho_h \|_{1,\Omega} \lesssim \|\eta_h\|_{1,\Omega} \lesssim
\|(\widetilde{\sigma}_{h,23}-\widetilde{\sigma}_{h,32},
\widetilde{\sigma}_{h,31}- \widetilde{\sigma}_{h,13},
\widetilde{\sigma}_{h,12}- \widetilde{\sigma}_{h,21})^{T} \|_{0,\Omega} \lesssim
\|\vect{\sigma}_h - \widetilde{\vect{\sigma}}_h\|_{0,\Omega}.
\end{equation*}
\end{enumerate}
To summarize, we obtain the desired $\widetilde{\vect{\tau}}_h = \vect{\rm
curl} \rho_h$ that satisfies \eqref{equ:symmetrize}.  This completes
the proof. 
\end{proof}

We are now in the position to prove the optimal $L^2$ error estimate. 
\begin{theorem} \label{thm:L2-stress} 
Assume that the Stokes pair
$\mathcal{P}_{k+2}-\mathcal{P}_{k+1}^{-1}$ is stable on
$\mathcal{T}_h$ and $k \geq d$. Assume further that the solution of
\eqref{equ:elasticity} satisfies $(\vect{\sigma},u)\in
H^{k+2}(\Omega;\mathbb{S}) \times H^{k+1}(\Omega;\mathbb{R}^d)$.
Then, the solution of the mixed LDG problem \eqref{equ:mixed-DG}
satisfies  
\begin{equation} \label{equ:sigma-l2}
\|\vect{\sigma} - \vect{\sigma}_h\|_{\mathcal{A},\Omega} \lesssim
h^{k+2}(|\vect{\sigma}|_{k+2,\Omega} + |u|_{k+1,\Omega}),
\end{equation}
where $\|\vect{\sigma}\|_{\mathcal{A},\Omega}^2 :=
(\mathcal{A}\vect{\sigma}, \vect{\sigma})_{\Omega}^{1/2}$.
\end{theorem}
\begin{proof}
By \eqref{equ:weak-form-sum}, \eqref{equ:row-BDM} and Lemma
\ref{lem:tilde-sigma}, we have that for any $v_h \in V_h$,
$$ 
\begin{aligned}
(f, v_h)_{\Omega} &= -(\vect{\sigma}_h, \vect{\varepsilon}_h(v_h))_{\Omega}+\langle
\widehat{\vect{\sigma}}_h n, v_h \rangle_{\partial \mathcal{T}_h}
= -(\vect{\sigma}_h, \vect{\nabla}_hv_h)_\Omega+ \langle
\widehat{\vect{\sigma}}_h n, v_h \rangle_{\partial \mathcal{T}_h} \\
& = -(\widetilde{\vect{\sigma}}_h, \vect{\nabla}_hv_h)_{\Omega} + \langle
\widetilde{\vect{\sigma}}_h n, v_h \rangle_{\partial \mathcal{T}_h} =
({\rm div}\widetilde{\vect{\sigma}}_h, v_h)_{\Omega}.
\end{aligned}
$$ 
By Lemma \ref{lem:symmetrize}, the symmetrized variable
$\vect{\sigma}_h^\star = \widetilde{\vect{\sigma}}_h +
\widetilde{\vect{\tau}}_h$ is piecewise $\mathcal{P}_{k+1}(K;\mathbb{S})$
and belongs to $H({\rm div},\Omega;\mathbb{S})$. Further, the
divergence-free of $\widetilde{\vect{\tau}}_h$ implies that 
\begin{equation}\label{equ:star-equ}
({\rm div}\vect{\sigma}_h^\star, v_h)_{\Omega} = (f, v_h)_{\Omega}.
\end{equation}
In \cite{hu2014family, hu2015family}, Hu and Zhang constructed the
conforming $\mathcal{P}_{k+1}-\mathcal{P}_k^{-1}$ mixed methods for linear
elasticity on simplicial grids when $k\geq d$. 
Hu also show that (cf. \cite[Remark 3.1]{hu2015finite}), when $k \geq
d$, there exists a projection $\Pi_h^c$ such that, 
\begin{subequations}
\begin{align}
({\rm div}(\vect{\tau} - \Pi_h^c\vect{\tau}), v_h)_{\Omega} &= 0 
~~~\qquad \qquad \qquad \forall
\vect{\tau}\in H^1(\Omega;\mathbb{S}), \label{equ:Pic-1} \\
\|\vect{\tau} - \Pi_h^c \vect{\tau}\|_{0,\Omega} &\lesssim h^{k+2}
|\vect{\tau}|_{k+2,\Omega} \qquad \forall \vect{\tau}\in
H^{k+2}(\Omega;\mathbb{S}). \label{equ:Pic-2} 
\end{align}
\end{subequations} 
By \eqref{equ:star-equ} and \eqref{equ:Pic-1}, we have 
$$ 
({\rm div}(\vect{\sigma}_h^\star - \Pi_h^c\vect{\sigma}), v_h) = 0
 \qquad \forall v_h \in V_h. 
$$ 
Taking $\vect{\tau}_h = \vect{\sigma}_h^\star - \Pi_h^c \vect{\sigma}$
in the error equation \eqref{equ:consistency}, we immediately have the
$\mathcal{A}$-orthogonality condition:
\begin{equation} \label{equ:orth} 
(\mathcal{A}(\vect{\sigma} - \vect{\sigma}_h), \vect{\sigma}_h^\star -
\Pi_h^c\vect{\sigma}) = 0.
\end{equation} 
Hence, by the energy estimate \eqref{equ:error-estimate},
\eqref{equ:row-estimate} and \eqref{equ:Pic-2},
$$ 
\begin{aligned}
\|\vect{\sigma} - \vect{\sigma}_h\|_{\mathcal{A},\Omega} &\leq
\|\vect{\sigma} - \Pi_h^{c}\vect{\sigma}\|_{\mathcal{A},\Omega} +
\|\vect{\sigma}_h^\star - \vect{\sigma}_h\|_{\mathcal{A},\Omega} \\ 
&\lesssim \|\vect{\sigma} - \Pi_h^{c}\vect{\sigma}\|_{0,\Omega} +
\|\widetilde{\vect{\tau}}_h\|_{0,\Omega} + \|\widetilde{\vect{\sigma}}_h -
\vect{\sigma}_h\|_{0,\Omega} \\ 
&\lesssim \|\vect{\sigma} - \Pi_h^{c}\vect{\sigma}\|_{0,\Omega} +
h^{1/2}\|(\widehat{\vect{\sigma}}_h - \vect{\sigma}_h)n\|_{\partial
\mathcal{T}_h} \\ 
&\lesssim \|\vect{\sigma} - \Pi_h^{c}\vect{\sigma}\|_{0,\Omega} +
h^{1/2}\|[\vect{\sigma}_h]\|_{\mathcal{E}_h^i} \\ 
&\lesssim \|\vect{\sigma} - \Pi_h^{c}\vect{\sigma}\|_{0,\Omega} + h
\|\vect{\sigma} - \vect{\sigma}_h\|_{*,\Omega} \\
&\lesssim h^{k+2}(|\vect{\sigma}|_{k+2,\Omega} + |u|_{k+1,\Omega}).
\end{aligned}
$$ 
This completes the proof.
\end{proof}
\begin{remark}
In the 2D case, the Scott-Vogelius elements
$\mathcal{P}_{k+2}-\mathcal{P}_{k+1}^{-1}$ are stable when $k \geq 2$
and the grid does not contain singular vertices (cf.
\cite{scott1985norm, guzman2019scott}).  Hence, in the 2D case, we have the
optimal $L^2$ estimate when $k \geq 2$ with some mild constrain
pertaining to the grids. 
\end{remark}

\section{Numerical examples} \label{sec:numerical}
In this section, we present some numerical results of the mixed LDG
method for linear elasticity problem.  The compliance tensor is given
by
$$
\mathcal{A}\vect{\sigma} = \frac{1}{2\mu}
\left(\vect{\sigma} - \frac{\lambda}{2\mu + d\lambda}
    \mathrm{tr}(\vect{\sigma})\vect{I}_d \right),
$$
where $\vect{I}_d$ is the $d\times d$ identity matrix. In the
computation, the Lam\'{e} constants are set to be $\mu = 1/2$ and
$\lambda = 1$. The parameter in \eqref{equ:blinear_a} is chosen
as $\eta_e = 1$ on all $e\in \mathcal E_h^i$. 

\paragraph{2D example.} The 2D problem is computed on the unit square
$\Omega = (0,1)^2$ with a homogeneous boundary condition that $u = 0$
on $\partial \Omega$. Let the exact solution be 
$$
u = 
\begin{pmatrix}
\mathrm{e}^{x-y} xy(1-x)(1-y) \\
\sin(\pi x)\sin(\pi y)
\end{pmatrix}.
$$
The exact stress function $\vect{\sigma}$ and the load function $f$
can be analytically derived from \eqref{equ:elasticity} and for a
given $u$.  Uniform grids with different grid sizes are adopted in
the computation.

\begin{table}[!htbp] 
\centering
\begin{subtable}{\textwidth} 
\centering
\caption{Linear elasticity: $\mathcal{P}_1^{-1} -
  \mathcal{P}_0^{-1}$, 2D uniform grids}
\begin{tabular}{c|cc|cc|cc}
\hline
$1/h$	& $\|u -u_h\|_{0,\Omega}$	&$h^n$	&$\|\vect{\sigma} -
\vect{\sigma}_h\|_{0,\Omega}$ & $h^n$
&$\|{\rm div}_h(\vect{\sigma} - \vect{\sigma}_h)\|_{0,\Omega}$	&$h^n$ \\ \hline
 4 & 0.135877 & ---  & 0.445892 & ---  & 3.839803 & ---  \\ 
 8 & 0.067302 & 1.01 & 0.177473 & 1.33 & 1.936584 & 0.99 \\ 
16 & 0.033543 & 1.00 & 0.080752 & 1.14 & 0.970346 & 1.00 \\ 
32 & 0.016757 & 1.00 & 0.039257 & 1.04 & 0.485431 & 1.00 \\ \hline 
\end{tabular}
\label{tab:elasticity-P1-P0-uniform}
\end{subtable}

\bigskip

\begin{subtable}{\textwidth}
\centering
\caption{Linear elasticity: $\mathcal{P}_2^{-1} - \mathcal{P}_1^{-1}$,
2D uniform grids}
\begin{tabular}{c|cc|cc|cc}
  \hline
$1/h$	& $\|u -u_h\|_{0,\Omega}$	&$h^n$	&$\|\vect{\sigma} -
\vect{\sigma}_h\|_{0,\Omega}$ & $h^n$
&$\|{\rm div}_h(\vect{\sigma} - \vect{\sigma}_h)\|_{0,\Omega}$	&$h^n$ \\ \hline
4  & 0.0198206 & ---  & 0.0425699 & ---  & 0.5850957 & ---  \\ 
8  & 0.0050264 & 1.98 & 0.0079777 & 2.42 & 0.1483264 & 1.98 \\ 
16 & 0.0012616 & 1.99 & 0.0017692 & 2.17 & 0.0372321 & 1.99 \\ 
32 & 0.0003158 & 2.00 & 0.0004284 & 2.05 & 0.0093191 & 2.00 \\ \hline
\end{tabular}
\label{tab:elasticity-P2-P1-uniform}
\end{subtable}

\bigskip

\begin{subtable}{\textwidth}
\centering 
\caption{Linear elasticity: $\mathcal{P}_3^{-1} - \mathcal{P}_2^{-1}$,
2D uniform grids}
\begin{tabular}{c|cc|cc|cc}
  \hline
$1/h$	& $\|u -u_h\|_{0,\Omega}$	&$h^n$	&$\|\vect{\sigma} -
\vect{\sigma}_h\|_{0,\Omega}$ & $h^n$
&$\|{\rm div}_h(\vect{\sigma} - \vect{\sigma}_h)\|_{0,\Omega}$	&$h^n$ \\ \hline
4  & 0.00217252 & ---  & 0.00341919 & ---  & 0.06370927 & ---  \\ 
8  & 0.00027548 & 2.98 & 0.00024533 & 3.80 & 0.00805005 & 2.98 \\ 
16 & 0.00003456 & 2.99 & 0.00001627 & 3.91 & 0.00100892 & 3.00 \\ 
32 & 0.00000432 & 3.00 & 0.00000104 & 3.96 & 0.00012620 & 3.00 \\ \hline
\end{tabular}
\label{tab:elasticity-P3-P2-uniform}
\end{subtable}
\caption{Linear elasticity: the convergence order for 2D example}
\label{tab:elasticity-uniform}
\end{table}

We list the errors and the rates of convergence of the computed
solution in Table \ref{tab:elasticity-uniform}.  The $(k+1)$-th order
convergence is observed for both the $L_2$ error of $u$ and the
$H_h({\rm div})$ error of $\vect{\sigma}$, which is in agreement with
Theorem \ref{thm:error-estimate}. Further, we see from Table
\ref{tab:elasticity-P3-P2-uniform} that $\|\vect{\sigma} -
\vect{\sigma}_h\|_{0,\Omega} = \mathcal{O}(h^4)$ when $k = 2$. This
convergence rate coincides with the statements in Theorem
\ref{thm:L2-stress}, which is also shown {\it sharp} from the $L^2$
errors of stress in Table
\ref{tab:elasticity-P1-P0-uniform}-\ref{tab:elasticity-P2-P1-uniform}. 

\paragraph{3D example.} Let the exact solution on the unit
cube be 
$$ 
u = 
\begin{pmatrix}
2^4 \\ 2^5 \\ 2^6
\end{pmatrix}
x(1-x)y(1-y)z(1-z).
$$ 
Again, the true stress function $\vect{\sigma}$ and the load function
$f$ are defined by the relations in \eqref{equ:elasticity}, for the
given solution $u$. In Table \ref{tab:elasticity-3D}, the errors and
the convergence order in various norms are listed when $k=0,1$. The
optimal orders of convergence are achieved respectively under the
$H_h({\rm div})$ norm for the stress and $L^2$ norm for the
displacement, which confirms Theorem \ref{thm:error-estimate}.

\begin{table}[!htbp] 
\centering
\begin{subtable}{\textwidth} 
\centering
\caption{Linear elasticity: $\mathcal{P}_1^{-1} -
  \mathcal{P}_0^{-1}$, 3D uniform grids}
\begin{tabular}{c|cc|cc|cc}
\hline
$1/h$	& $\|u -u_h\|_{0,\Omega}$	&$h^n$	&$\|\vect{\sigma} -
\vect{\sigma}_h\|_{0,\Omega}$ & $h^n$
&$\|{\rm div}_h(\vect{\sigma} - \vect{\sigma}_h)\|_{0,\Omega}$	&$h^n$ \\ \hline
 2 & 0.235741 & ---  & 1.221265 & ---  & 7.534218 & ---  \\ 
 4 & 0.127481 & 0.89 & 0.536012 & 1.19 & 4.420875 & 0.77 \\ 
 8 & 0.063704 & 1.00 & 0.210303 & 1.35 & 2.294909 & 0.95 \\ \hline 
\end{tabular}
\label{tab:elasticity-3D0}
\end{subtable}

\bigskip

\begin{subtable}{\textwidth}
\centering
\caption{Linear elasticity: $\mathcal{P}_2^{-1} - \mathcal{P}_1^{-1}$,
3D uniform grids}
\begin{tabular}{c|cc|cc|cc}
  \hline
$1/h$	& $\|u -u_h\|_{0,\Omega}$	&$h^n$	&$\|\vect{\sigma} -
\vect{\sigma}_h\|_{0,\Omega}$ & $h^n$
&$\|{\rm div}_h(\vect{\sigma} - \vect{\sigma}_h)\|_{0,\Omega}$	&$h^n$ \\ \hline
 2 & 0.0831048 & ---  & 0.3641751 & ---  & 2.8564400 & ---  \\ 
 4 & 0.0227446 & 1.87 & 0.0664638 & 2.45 & 0.7833919 & 1.87 \\ 
 8 & 0.0058207 & 1.97 & 0.0123827 & 2.42 & 0.2007023 & 1.96 \\ \hline
\end{tabular}
\label{tab:elasticity-3D1}
\end{subtable}
\caption{Linear elasticity: the convergence order for 3D example}
\label{tab:elasticity-3D}
\end{table}

\section{Concluding remarks} \label{sec:concluding}
In this paper, we present the first a priori error analysis of mixed
DG method for solving the linear elasticity problem.  We provide numerical
evidence indicating the sharpness of our estimates, namely, the
convergence order of $k+1$ both stress in $H_h({\rm div})$-norm and
displacement in $L^2$-norm with the elements pair $(\vect{\sigma}_h,
u_h)\in \vect{\Sigma}_h^{k+1}\times V_h^k$. The estimate holds for any   
$k\geq 0$ in arbitrary dimension, making the MDG more meaningful for
the linear elasticity as the lower order conforming
$\mathcal{P}_{k+1}$-$\mathcal{P}_k^{-1}$ elasticity element does
not exist on general simplicial grids \cite{wu2017interior}.  Since
there is a close connection between elasticity elements and Stokes
elements, we also prove the optimal $L^2$ error estimate for the
stress provided that the $\mathcal{P}_{k+2}-\mathcal{P}_{k+1}^{-1}$
Stokes pair is stable and $k \geq d$.

\bibliographystyle{plain}
\bibliography{mixedDG}

\end{document}